\documentclass[11pt]{article}
\usepackage[T1]{fontenc}
\usepackage[utf8]{inputenc}
\usepackage{lmodern,microtype}
\usepackage[margin=1in]{geometry}
\usepackage{amsmath,amssymb,amsthm,mathtools}
\usepackage{xurl,xcolor,listings}
\usepackage[hidelinks]{hyperref}
\hypersetup{pdftitle={Entropy bounds and global couplings for union-closed families},pdfauthor={Yunjiang Jiang}}
\newtheorem{theorem}{Theorem}[section]
\newtheorem{lemma}[theorem]{Lemma}
\newtheorem{proposition}[theorem]{Proposition}
\theoremstyle{remark}

\newcommand{\F}{\mathcal F}
\newcommand{\E}{\mathbb E}

\newcommand{\ind}{\mathbf 1}

\allowdisplaybreaks[1]
\title{Entropy bounds and global couplings\\for union-closed families}
\author{Yunjiang Jiang\thanks{AI research assistance: GPT-6.}}
\date{}
\begin{document}
\maketitle
\begin{abstract}
We give a computer-assisted proof that every finite union-closed family
containing a nonempty set has an element in at least
\(0.38288525\) of its members. The argument combines independent sampling
with conditionally independent sampling and an explicit product lower
bound for a binary-entropy kernel. We identify the limiting constant of
this pointwise method through a one-variable stationary equation.
We also construct a global coupling by balancing inverse union
multiplicities and prove a strictly positive, quantitative entropy gain
over independent sampling. Combining the two arguments gives an
additional frequency bound depending on the size of the family.
\end{abstract}

\section{Introduction}\label{sec:statement}

A family \(\F\subseteq 2^{[n]}\), where \([n]=\{1,\ldots,n\}\), is
\emph{union-closed} if \(A\cup B\in\F\) whenever \(A,B\in\F\).
Its members are distinct sets. Write \(N=|\F|\), and define the frequency
of element \(i\) by
\[
p_i=\frac{|\{A\in\F:i\in A\}|}{N},\qquad
p=\max_{1\le i\le n}p_i.
\]
Frankl's conjecture asserts \(p\ge1/2\) whenever \(\F\) contains a
nonempty set. The nonempty-set condition excludes \(\{\varnothing\}\).
For example, in
\(\{\varnothing,\{1\},\{2\},\{1,2\}\}\), both frequencies are \(1/2\).

Gilmer~\cite{gilmer} introduced the entropy approach and proved the
first absolute constant bound, \(0.01\). The central observation is
that independent samples from a family with sufficiently small element
frequencies have a union with greater entropy than either sample.
For a uniform distribution on a union-closed family, this contradicts
the upper bound on the number of possible unions.
The subsequent work of Alweiss, Huang, and Sellke~\cite{ahs},
Chase and Lovett~\cite{chase-lovett}, Sawin~\cite{sawin}, and
Pebody~\cite{pebody} established the constant
\((3-\sqrt5)/2\).

Sawin~\cite{sawin} also proposed using dependent samples to pass beyond
this independent-sampling threshold.
Yu~\cite{yu} developed finite-dimensional optimization bounds and
numerically evaluated the resulting improvement at about \(0.38234\).
Cambie~\cite{cambie} studied dependent sampling and the limits of this
entropy approach. Liu~\cite{liu} introduced conditionally IID
couplings, proved a further strict improvement, and obtained a value
near \(0.38271\) under numerically verified hypotheses.
Our construction uses this line of work, together with the auxiliary
profile developed by Costa~\cite{costa}.

\begin{theorem}[Frequency bound]\label{thm:main}
For every finite union-closed family containing a nonempty set,
\[
\boxed{\displaystyle p\ge\frac{1531541}{4000000}=0.38288525.}
\]
The proof uses the interval verification described in
Section~\ref{sec:certificates}; the complete program is included in
Appendix~\ref{app:code}.
\end{theorem}

We work on a finite ground set, as in the usual formulation. Even if
a finite family is initially described on an arbitrary ground set, it
can be reduced to this case: choose an element of a nonempty member and,
for each pair of distinct members, one element of their symmetric
difference. Intersecting every member with the finite set of chosen
elements is injective, preserves unions, and preserves the frequencies
of the chosen elements.

Here is the idea of the proof. Choose a member \(X\) of \(\F\) uniformly,
representing it by its \(0\)-\(1\) indicator vector. It has entropy
\(H(X)=\ln N\). We construct two different ways to sample a pair of
members and take their union. Union closure says that each resulting
union still takes at most \(N\) values, so its entropy is at most
\(\ln N\). On the other hand, our inequalities force a weighted average
of those union entropies to be at least
\[
\frac{1-p}{m}\ln N,\qquad m=0.61711475.
\]
When \(N\ge2\), comparing the two bounds gives \(1-p\le m\).
When \(N=1\), the sole member is nonempty, so some element has frequency
one. Thus the entire problem is to justify the entropy lower bound.

\paragraph{Why two sampling rules?}
Independent samples make the conditional probabilities easy to average,
but their entropy estimate alone stops below the desired constant.
The second rule shares fresh randomness between the samples, while
keeping each sample uniform on the family. The resulting dependence
can increase the entropy of the union. Its conditional probabilities
are harder to average, so we bound its binary-entropy kernel below by
a product \(\psi(x)\psi(y)\). Conditional independence given the
shared history then turns the expectation of this product into a
square, and nonnegativity of variance gives the needed lower bound.
A weighted combination of the two rules completes the comparison.

\paragraph{How the proof is organized.}
There are three tasks. First, the entropy reduction shows that two
explicit inequalities on \([0,1]^2\) imply the frequency bound.
Second, we construct the sampling coefficient and product profile and
prove those inequalities: analytic estimates cover the boundary
regions, and the program in the appendix covers the remaining compact
domains. Third, a separate argument adjusts the probabilities of whole
pairs of sets to improve on independent sampling. This last argument
is entirely analytic and yields an additional positive correction for
each finite family size. It does not enlarge the universal constant,
because the stated size-only correction tends to zero.

The constant in Theorem~\ref{thm:main} is an exact rational choice, which
we denote by \(c_0=1531541/4000000\). Its origin is an optimization
in the entropy argument. Let \(q_*\) be the unique solution in \((0,1)\) of
\begin{equation}\label{eq:stationary}
(1-q)\ln(1-q)+q\ln(1+q)=0.
\end{equation}
The pointwise method has the upper limit
\begin{equation}\label{eq:constant-ceiling}
C_*=\frac12-\frac{(1-q_*)\ln(1+q_*)}{2\ln2}
\approx0.3828852599667978.
\end{equation}
We prove this limitation in Section~\ref{sec:ceiling}.
The achieved bound \(c_0\) lies just below \(C_*\); the theorem does not
assert that \(c_0\) is a root of \eqref{eq:stationary}, or that \(C_*\)
itself is achieved. The latter is a limit of the proof method, not
an upper bound on the frequency in Frankl's conjecture.
The decimal approximations to \(q_*\) and \(C_*\) are given for
orientation; their definitions are the exact equations
\eqref{eq:stationary}--\eqref{eq:constant-ceiling}.
No root approximation is used in the proof of the achieved bound:
the auxiliary profile below has an exact rational center.

The sampling construction is adapted from Liu~\cite{liu}, and the
starting auxiliary profile is adapted from Costa~\cite{costa}.
We use a rational profile center and a strengthened upper tail.
All analytic reductions and boundary estimates are included below.
The complete interval verification is included in Appendix~\ref{app:code}.

The second argument concerns a coupling of entire sets. If \(N=|\F|\ge2\),
let \(r(S)\) be the number of ordered pairs in \(\F^2\) whose union is \(S\).
We choose positive factors \(a(A)\) so that
\[
P_*(A,B)=\frac{a(A)a(B)}{r(A\cup B)}
\]
has uniform marginals on \(\F\). If \(\nu_*\) and \(\nu_0\) denote the
union laws under this coupling and independent sampling, respectively,
we prove
\[
H(\nu_*)-H(\nu_0)\ge
\frac{\ln^2(2N-1)}{64N^2\ln N}>0.
\]
The construction and this entropy estimate are entirely analytic.

The exposition proceeds from the entropy reduction to the scalar
optimization, the auxiliary functions, and their inequalities.
Section~\ref{sec:global} develops the global coupling and its consequence
for element frequencies. The appendices give the derivative formulas
and the complete verification program.

\section{Entropy preliminaries}\label{sec:entropy}

For a finite-valued random variable \(W\), with probabilities \(q_w\), set
\[
H(W)=-\sum_w q_w\ln q_w,\qquad 0\ln0:=0.
\]
For a bit with probability \(x\) of being zero, its entropy is
\[
h(x)=-x\ln x-(1-x)\ln(1-x).
\]
Thus \(h(0)=h(1)=0\), \(h(x)=h(1-x)\), and
\begin{equation}\label{eq:hderivatives}
h'(x)=\ln\frac{1-x}{x},\qquad
h''(x)=-\frac{1}{x(1-x)}<0\quad(0<x<1).
\end{equation}
In particular, \(h\) is concave, increases up to \(1/2\), decreases
thereafter, and has maximum \(\ln2\).

\paragraph{Maximum entropy.}
If \(W\) takes \(r\le N\) possible values with positive probabilities,
concavity of the logarithm gives
\[
H(W)=\sum_w q_w\ln\frac1{q_w}
\le\ln\left(\sum_wq_w\frac1{q_w}\right)=\ln r\le\ln N.
\]
Equality holds for the uniform distribution on \(N\) values.

\paragraph{Conditional entropy and the chain rule.}
Define \(H(V\mid W)=\sum_w\Pr(W=w)H(V\mid W=w)\).
Writing each joint probability as a marginal times a conditional
probability, and taking logarithms, proves
\[
H(V,W)=H(W)+H(V\mid W).
\]
Applying this repeatedly to a vector gives
\begin{equation}\label{eq:chain}
H(X)=\sum_{i=1}^n H(X_i\mid X_{<i}),\qquad
X_{<i}=(X_1,\ldots,X_{i-1}).
\end{equation}

\paragraph{Conditioning reduces entropy.}
The function \(-q\ln q\) is concave, so entropy is a concave function
of the probability vector. Averaging the more detailed conditional
distributions and applying concavity gives
\[
H(V\mid W)\ge H(V\mid W,T).
\]
Consequently, if \(W\) is a deterministic function of \(T\),
then \(H(V\mid W)\ge H(V\mid T)\). Also,
if \(V\) is a deterministic function of \(T\), the chain rule gives
\(H(V)\le H(T)\).

\paragraph{Two elementary binary-entropy inequalities.}
Concavity and \(h(0)=0\) imply
\begin{equation}\label{eq:concavity-scaling}
h(cx)\ge c\,h(x)\quad(0\le c,x\le1).
\end{equation}
If two independent bits each have probability \(x\) of being one,
their AND has probability \(x^2\) of being one. Since AND is a
deterministic function of the pair, the preceding facts imply
\begin{equation}\label{eq:and}
h(x^2)\le2h(x).
\end{equation}
These are the only general entropy facts needed outside the sampling
argument.

All logarithms are natural. Entropy in base \(2\) is obtained by dividing
by \(\ln2\); this rescales both sides of the entropy comparisons by the
same factor.

\section{Reduction to two scalar inequalities}\label{sec:reduction}

We give the entropy reduction in a general form, following
Gilmer's entropy-growth strategy~\cite{gilmer} and Liu's conditionally
IID construction~\cite{liu}. This specifies precisely what the
subsequent scalar inequalities must establish.
Let \(a:[0,1]\to[0,1]\) be a function and define
\begin{equation}\label{eq:kernel-general}
r(x,y)=xy+a(x)a(y)\bigl(\min(x,y)-xy\bigr),\qquad
K(x,y)=h(r(x,y)).
\end{equation}
Let \(\psi:[0,1]\to[0,\infty)\), and let \(w_0,w_4\ge0\) sum to one.
The subscript \(4\) records the connection with Liu's Example 4.

\begin{proposition}[Entropy reduction]\label{prop:reduction}
Suppose \(m>0\), and, for every \(x,y\in[0,1]\),
\begin{align}
K(x,y)&\ge\psi(x)\psi(y),\tag{K}\label{eq:K}\\
2m\bigl[w_0h(xy)+w_4\psi(x)\psi(y)\bigr]
&\ge xh(y)+yh(x).\tag{M}\label{eq:M}
\end{align}
Then every finite union-closed family containing a nonempty set
has \(p\ge1-m\).
\end{proposition}

\subsection{The native conditional probabilities}
Let \(X\) be uniform on \(\F\).
For each possible prefix \(\alpha\), define
\[
x_i(\alpha)=\Pr(X_i=0\mid X_{<i}=\alpha).
\]
A prefix with probability zero can be assigned any value, since it
will never be reached by the constructions below. The term
\emph{native} means that this probability is computed from the prescribed law of
\(X\), before choosing any coupling of the two samples.

Let \(P_i\) be the distribution of \(x_i(X_{<i})\). By conditioning,
\begin{equation}\label{eq:Pi}
\E_{P_i}x=1-p_i,\qquad
\E_{P_i}h(x)=H(X_i\mid X_{<i}).
\end{equation}
Both constructions will retain the same individual sample law, so
the same \(P_i\) applies to both.

\subsection{The independent rule}
Generate two independent samples \(X^{(0)},Y^{(0)}\), each with the
law of \(X\), and let \(Z_0=X^{(0)}\lor Y^{(0)}\), where \(\lor\)
denotes coordinatewise OR.
Given their prefixes at coordinate \(i\), their zero probabilities
are \(S=x_i(X^{(0)}_{<i})\) and \(T=x_i(Y^{(0)}_{<i})\).
They have joint distribution \(P_i\otimes P_i\). The union bit is
zero exactly when both bits are zero, so
\[
H((Z_0)_i\mid X^{(0)}_{<i},Y^{(0)}_{<i})
=\E_{P_i\otimes P_i}h(ST).
\]
The union prefix is determined by those two prefixes. Conditioning
therefore gives
\begin{equation}\label{eq:ind-lower}
H((Z_0)_i\mid (Z_0)_{<i})
\ge\E_{P_i\otimes P_i}h(ST).
\end{equation}

\subsection{The shared-randomness rule}
At coordinate \(i\), draw a fresh shared uniform random variable
\(U_i\) on \([0,1]\). For a native zero probability \(x\), set
\begin{equation}\label{eq:Q}
Q_{U_i}(x)=x+a(x)\bigl(\ind_{\{U_i\le x\}}-x\bigr).
\end{equation}
Conditional on \(U_i\) and the two prefixes, use independent local
randomness to generate two bits with zero probabilities \(Q_{U_i}(S)\)
and \(Q_{U_i}(T)\). Repeat this at each coordinate to obtain
\(X^{(4)},Y^{(4)}\), and put \(Z_4=X^{(4)}\lor Y^{(4)}\).

There are three facts to check.
First, \(Q_U(x)\) is a convex combination of \(x\) and an indicator,
so it lies in \([0,1]\). Second,
\[
\E_UQ_U(x)=x.
\]
The fresh \(U_i\) is independent of either past prefix. Thus,
averaging over \(U_i\), each sample has its native transition
probability at every coordinate. Induction proves that each
complete sample has the law of \(X\).

Third, writing \(I_x=\ind_{\{U\le x\}}\), we have
\(\E I_x=x\) and \(\E I_xI_y=\min(x,y)\). Expanding the product in
\eqref{eq:Q} gives
\begin{equation}\label{eq:Qproduct}
\E_U[Q_U(x)Q_U(y)]
=xy+a(x)a(y)\bigl(\min(x,y)-xy\bigr)=r(x,y).
\end{equation}
It follows that, conditional on the two prefixes, the union bit
has zero probability \(r(S,T)\), and hence entropy \(K(S,T)\).

\subsection{Where conditional independence is used}
Let \(V_i=(U_1,\ldots,U_{i-1})\) be the shared history.
Conditional on \(V_i\), each prefix is a function of that fixed
history and its own independent local random variables. The two
prefixes are therefore independent and identically distributed
conditional on \(V_i\). The same is true of
\(S=x_i(X^{(4)}_{<i})\) and \(T=x_i(Y^{(4)}_{<i})\).
Consequently,
\begin{align}
\E[\psi(S)\psi(T)]
&=\E\left[\bigl(\E[\psi(S)\mid V_i]\bigr)^2\right]\notag\\
&\ge\bigl(\E\psi(S)\bigr)^2
=\bigl(\E_{P_i}\psi\bigr)^2. \label{eq:conditional-jensen}
\end{align}
The inequality is simply nonnegativity of variance.
Combining conditioning, \eqref{eq:Qproduct}, and \eqref{eq:K} gives
\begin{equation}\label{eq:rule4-lower}
H((Z_4)_i\mid (Z_4)_{<i})
\ge \E K(S,T)\ge\bigl(\E_{P_i}\psi\bigr)^2.
\end{equation}

\subsection{Completing the reduction}
Integrate \eqref{eq:M} against the \emph{independent} product
distribution \(P_i\otimes P_i\). Its right side becomes
\(2(\E_{P_i}x)(\E_{P_i}h)\). Dividing by \(2m\) gives
\[
w_0\E_{P_i\otimes P_i}h(ST)
+w_4\bigl(\E_{P_i}\psi\bigr)^2
\ge\frac{1-p_i}{m}H(X_i\mid X_{<i}).
\]
Equations \eqref{eq:ind-lower} and \eqref{eq:rule4-lower}, followed
by summation and the chain rule, now imply
\begin{align*}
w_0H(Z_0)+w_4H(Z_4)
&\ge\sum_{i=1}^n\frac{1-p_i}{m}H(X_i\mid X_{<i})\\
&\ge\frac{1-p}{m}H(X).
\end{align*}
Every conditional entropy is nonnegative, justifying the second
inequality. Union closure makes both \(Z_0\) and \(Z_4\) take values
in \(\F\). Their weighted entropy is therefore at most \(\ln N\).
For \(N\ge2\), divide by \(H(X)=\ln N>0\) to obtain \(p\ge1-m\).
The case \(N=1\) was handled in Section~\ref{sec:statement}.
This proves Proposition~\ref{prop:reduction}.

The two rules may be constructed on entirely separate probability
spaces. Only their entropy values are averaged; the argument does
not require choosing a new rule independently at every coordinate.

\section{The constant arising from the pointwise method}
\label{sec:ceiling}

Two necessary conditions, one at the boundary and one on the diagonal,
identify the constant \(C_*\) in \eqref{eq:constant-ceiling}.

\begin{proposition}[Limitation of the pointwise method]
\label{prop:sharper-ceiling}
If \eqref{eq:K}--\eqref{eq:M} hold, then
\[
1-m\le C_*:=
\frac12-\max_{0<q<1}
\frac{qh(q)-\frac12h(q^2)}{\ln2}.
\]
The maximum occurs at the unique solution of
\eqref{eq:stationary}, and \(C_*\) is given by
\eqref{eq:constant-ceiling}.
\end{proposition}
\begin{proof}
The sampling kernel satisfies \(K(1,1)=h(1)=0\).
Thus \eqref{eq:K} forces \(\psi(1)=0\).
Put \(y=1\) in \eqref{eq:M}, and choose any \(0<x<1\). It gives
\[
2mw_0h(x)\ge h(x),
\qquad\text{hence}\qquad mw_0\ge\frac12.
\tag{E}
\label{eq:endpoint-barrier}
\]
Next put \(x=y=q\) in \eqref{eq:M}. Since
\(\psi(q)^2\le K(q,q)\le\ln2\), and \(w_4=1-w_0\), we obtain
\[
\begin{aligned}
qh(q)
&\le m\bigl[w_0h(q^2)+w_4\psi(q)^2\bigr]\\
&\le m\ln2-mw_0\bigl[\ln2-h(q^2)\bigr]\\
&\le m\ln2-\frac12\bigl[\ln2-h(q^2)\bigr].
\end{aligned}
\]
The last step uses (E) and \(h(q^2)\le\ln2\).
Rearranging gives, for every \(q\in(0,1)\),
\[
1-m\le\frac12-\frac{qh(q)-\frac12h(q^2)}{\ln2}.
\tag{C}
\label{eq:ceiling-each-q}
\]
Taking the strongest of these bounds proves the claimed formula.

To identify the maximizing point, put \(f(q)=qh(q)-h(q^2)/2\).
Direct differentiation and cancellation give
\[
f'(q)=-(1-q)\ln(1-q)-q\ln(1+q),
\]
\[
f''(q)=\ln\frac{1-q}{1+q}+\frac1{1+q}.
\]
For \(0<q\le1/2\), the inequalities
\(-\ln(1-q)>q\) and \(\ln(1+q)<q\) show that
\(f'(q)>q(1-2q)\ge0\).
For \(q\ge1/2\), \(f''(1/2)=-\ln3+2/3<0\), and
\[
f'''(q)=-\frac2{1-q^2}-\frac1{(1+q)^2}<0.
\]
Thus \(f'\) is strictly decreasing on \([1/2,1)\).
It starts there at \(\frac12\ln(4/3)>0\) and tends to
\(-\ln2<0\). There is exactly one maximizing point \(q_*\).
The equation \(f'(q_*)=0\) is precisely
\eqref{eq:stationary}. Expanding \(h(q^2)\) gives
\[
f(q)=\frac12(1-q)^2\ln(1-q)
     +\frac12(1-q^2)\ln(1+q).
\]
At the stationary point,
\((1-q_*)\ln(1-q_*)=-q_*\ln(1+q_*)\), so this simplifies to
\[
f(q_*)=\frac12(1-q_*)\ln(1+q_*).
\]
Substitution proves \eqref{eq:constant-ceiling}.
For numerical evaluation, put
\(\Phi(q)=(1-q)\ln(1-q)+q\ln(1+q)\).
Ordinary Newton iteration is
\[
q_{j+1}=q_j-
\frac{\Phi(q_j)}
{\displaystyle\ln\frac{1+q_j}{1-q_j}-\frac1{1+q_j}}.
\]
Starting at \(q_0=0.69\) gives
\(q_*\approx0.6909077387254078\), and substitution into
\eqref{eq:constant-ceiling} gives the stated decimal value of \(C_*\).
No formal certification of these decimal approximations is needed:
the formulas for \(q_*\) and \(C_*\) are exact.
\end{proof}

The same obstruction applies if the dependent term is replaced by a
finite weighted sum of admissible products
\(\psi_j(x)\psi_j(y)\): each \(\psi_j(1)=0\), and each diagonal
product is at most \(\ln2\). The identical endpoint and diagonal
argument then applies.

This limitation applies to the pointwise product reduction
\eqref{eq:K}--\eqref{eq:M}. It does not exclude stronger estimates from
other entropy arguments or from global couplings. Our parameter choice
in the next section yields \(1-m=c_0<C_*\).

\section{The explicit auxiliary functions}\label{sec:profile}

We now specify the functions in \eqref{eq:K}--\eqref{eq:M}.
Every terminating decimal below denotes an exact rational number.
Set
\begin{equation}\label{eq:parameters}
\begin{gathered}
m=\frac{2468459}{4000000}=0.61711475,\qquad
w_0=\frac{8102221}{10000000},\qquad w_4=\frac{1897779}{10000000},\\
A=\frac{179}{250}=0.716,\qquad B=\frac{721}{1000}=0.721,\\
z=\frac{181}{256},\qquad \sigma=\frac{99999999}{100000000}.
\end{gathered}
\end{equation}
The letter \(z\) in this section is a fixed breakpoint, not a random
union vector. The notation \(w_4\) refers to a sampling rule, not a
fourth power.

Choose the exact rational center
\[
x_*=\frac{172727}{250000}=0.690908.
\]
It lies strictly between \(0.69\) and \(0.70\).
This is a fixed parameter, not a rounded value used in place of an
algebraic root. Every sign check below uses this exact rational number.

The center is chosen close to the optimizer \(q_*\) from
\eqref{eq:stationary}. The sampling weights satisfy \(mw_0>1/2\),
as required by the endpoint constraint
\eqref{eq:endpoint-barrier}. These are fixed parameters of the
construction; no optimality assertion is needed.
Define
\begin{equation}\label{eq:E}
E(x)=-\frac{h'(x_*)}{2h(x_*)}(x-x_*)-\frac7{10}(x-x_*)^2,
\end{equation}
and the continuous function
\begin{equation}\label{eq:D}
D(x)=
\begin{cases}
h(x),&0\le x\le1/2,\\
\ln2,&1/2\le x\le z,\\
\dfrac{\ln2}{h(z^2)}h(x^2),&z\le x\le1.
\end{cases}
\end{equation}
Continuity at \(1/2\) follows from \(h(1/2)=\ln2\), and continuity
at \(z\) follows by cancellation of \(h(z^2)\).

The first profile is
\begin{equation}\label{eq:phi}
\varphi(x)=\sigma\sqrt{\frac{D(x)h(x)}{h(x_*)}}\,e^{E(x)}.
\end{equation}
It is continuous, nonnegative, and vanishes at 0 and 1.
This is the profile before the upper-tail improvement. In particular,
it has no artificial cutoff near either endpoint.

Set
\begin{equation}\label{eq:s}
s(x)=\sqrt{h(x^2)},\qquad
\theta(x)=\frac{x-A}{B-A}.
\end{equation}
Our stronger profile is
\begin{equation}\label{eq:psi}
\psi(x)=
\begin{cases}
\varphi(x),&0\le x\le A,\\
(1-\theta(x))\varphi(x)+\theta(x)s(x),&A\le x\le B,\\
s(x),&B\le x\le1.
\end{cases}
\end{equation}
The interpolation ensures continuity: its weight on \(s\) is zero
at \(A\) and one at \(B\). Thus \(\psi\) is continuous, nonnegative,
and also vanishes at both endpoints.

Finally, choose the sampling coefficient
\begin{equation}\label{eq:a}
a(x)=
\begin{cases}
1,&0\le x\le1/2,\\
\sqrt{\dfrac{1-2x^2}{2x(1-x)}},&1/2<x<1/\sqrt2,\\
0,&1/\sqrt2\le x\le1.
\end{cases}
\end{equation}
On the middle interval the numerator and denominator are positive,
and the ratio is at most one because \(1\le2x\).
The endpoint values agree, so \(a\) is continuous and takes values
in \([0,1]\). We use it in \eqref{eq:kernel-general}.

The purpose of these choices is concrete. In the upper region
\(x\ge A>1/\sqrt2\), the coefficient \(a(x)\) is zero, so the kernel
reduces to \(h(xy)\). The next section shows that this simpler kernel
supports the stronger product \(s(x)s(y)\). A one-variable bridge
will allow the stronger tail to coexist with the original lower part.

\section{An upper-tail extension}\label{sec:tail}

\subsection{Geometric concavity of binary entropy}
Ordinary concavity compares \(h\) at arithmetic averages.
We need a related fact that compares it at geometric averages.

\begin{lemma}\label{lem:geometric}
For \(0<x,y<1\),
\[
h(xy)\ge\sqrt{h(x^2)h(y^2)}=s(x)s(y).
\]
Moreover, for fixed \(0<y\le A\), the ratio \(h(xy)/s(x)\)
is nondecreasing for \(A\le x<1\).
\end{lemma}

\begin{proof}
Write
\[
q(x)=\frac{xh'(x)}{h(x)},\qquad \ell(x)=-\ln(1-x).
\]
The identity \(h(x)-xh'(x)=\ell(x)\), together with
\eqref{eq:hderivatives}, gives
\begin{align*}
h(x)^2q'(x)
&=h(x)h'(x)+xh(x)h''(x)-x(h'(x))^2\\
&=h'(x)\ell(x)-\frac{h(x)}{1-x}.
\end{align*}
For \(x\ge1/2\), the first term is nonpositive and the second
strictly negative. For \(x<1/2\), substitute \(h=xh'+\ell\) to obtain
\[
h(x)^2q'(x)
=h'(x)\left(\ell(x)-\frac{x}{1-x}\right)
-\frac{\ell(x)}{1-x}<0.
\]
Here \(h'(x)>0\) and
\[
\ell(x)=\int_0^x\frac{dt}{1-t}<\frac{x}{1-x}.
\]
Thus \(q\) is strictly decreasing.

For \(f(t)=\ln h(e^t)\), with \(t<0\), one has
\(f'(t)=q(e^t)\) and \(f''(t)=e^tq'(e^t)<0\).
Apply concavity of \(f\) at \(2\ln x\) and \(2\ln y\):
\[
\ln h(xy)\ge\frac12\ln h(x^2)+\frac12\ln h(y^2).
\]
Exponentiation proves the first assertion.
For the second,
\[
\frac{d}{dx}\ln\frac{h(xy)}{s(x)}
=\frac{q(xy)-q(x^2)}{x}\ge0,
\]
because \(y\le A\le x\) implies \(xy\le x^2\).
\end{proof}

\subsection{The two facts needed at the join}
We will prove the following bridge inequality in
Section~\ref{sec:bridge}:
\begin{equation}\label{eq:bridge}
h(Ay)\ge s(A)\varphi(y)\qquad(0\le y\le A).
\end{equation}
We also need \(\varphi(x)\le s(x)\) for \(x\ge A\).
Because \(A>z\), the formula for their ratio simplifies to
\[
\frac{\varphi(x)}{s(x)}
=\sigma\sqrt{\frac{\ln2\,h(x)}{h(z^2)h(x_*)}}\,e^{E(x)}
\quad(A\le x<1).
\]
Its logarithmic derivative is
\[
\frac{h'(x)}{2h(x)}+E'(x).
\]
This derivative is strictly decreasing: its derivative is
\[
\frac{h''(x)h(x)-(h'(x))^2}{2h(x)^2}-\frac75<0.
\]
The scalar interval checks give
\begin{equation}\label{eq:tail-scalars}
\begin{gathered}
s(A)-\varphi(A)>0,\\
-\left(\frac{h'(A)}{2h(A)}+E'(A)\right)>0.
\end{gathered}
\end{equation}
Thus the ratio starts below one and decreases. By continuity the
endpoint \(x=1\) is included, and
\begin{equation}\label{eq:profile-order}
\varphi(x)\le\psi(x)\le s(x)\qquad(A\le x\le1).
\end{equation}
On \([0,A]\), of course, \(\psi=\varphi\). In particular,
\(\psi\ge\varphi\) on the entire unit interval.

\subsection{Passing from the original to the stronger kernel bound}
Assume for the moment the original profile bound
\begin{equation}\label{eq:original-K}
K(x,y)\ge\varphi(x)\varphi(y)\qquad(0\le x,y\le1).
\end{equation}
It is proved independently in the next section.
We verify \eqref{eq:K} by dividing the square into three cases.

If \(x,y\le A\), equation \eqref{eq:original-K} is exactly the
required result. If \(x,y\ge A\), then \(a(x)=a(y)=0\), so
Lemma~\ref{lem:geometric} and \eqref{eq:profile-order} give
\[
K(x,y)=h(xy)\ge s(x)s(y)\ge\psi(x)\psi(y).
\]
In the mixed case, by symmetry take \(x\ge A\ge y\).
Again \(a(x)=0\). Monotonicity of the ratio in
Lemma~\ref{lem:geometric} and the bridge give
\[
K(x,y)=h(xy)
\ge s(x)\frac{h(Ay)}{s(A)}
\ge s(x)\varphi(y)\ge\psi(x)\psi(y).
\]
The reasoning for ratios used interior points, but the final
inequalities extend to every boundary point by continuity.
This proves \eqref{eq:K} once \eqref{eq:original-K} and
\eqref{eq:bridge} have been established.

\section{The kernel and bridge inequalities}\label{sec:kernel}

We divide the original kernel verification into a compact square
and two boundary strips. Set
\[
L=\frac1{16},\qquad U=\frac{15}{16}.
\]
The interval verification proves \eqref{eq:original-K} on \([L,U]^2\).
The estimates below cover every pair outside that square.

\subsection{Constants and elementary bounds}
For brevity, write \(h_*=h(x_*)\). Define
\begin{align}
C_b&=\frac{\sigma}{\sqrt{h_*}}e^{E(L)},&
P_0&=\frac{\sigma\ln2}{\sqrt{h_*}}e^{E(1/2)},\notag\\
C_g&=\sigma\sqrt{\frac{2\ln2}{h_*h(z^2)}}e^{E(1)},&
C_t&=\frac{\sigma}{U}
\sqrt{\frac{\ln2\,h(U^2)h(U)}{h(z^2)h_*}}e^{E(1)}.
\label{eq:strip-constants}
\end{align}
The scalar interval checks give
\begin{equation}\label{eq:strip-checks}
\begin{gathered}
E'(1)>0,\qquad
C_bP_0<\frac12,\qquad 2C_bP_0<1,\\
C_gC_t<1,\\
1-E'(1/2)-\frac{h'(z^2)}{h(z^2)}>0.
\end{gathered}
\end{equation}
Because \(E''=-7/5<0\), the minimum of \(E'\) on \([0,1]\) is at
one. Hence \(E\) is increasing throughout the interval.

For \(x\le L\), equation \eqref{eq:D} gives \(D(x)=h(x)\), so
\begin{equation}\label{eq:lower-profile}
\varphi(x)=\frac{\sigma}{\sqrt{h_*}}h(x)e^{E(x)}
\le C_bh(x).
\end{equation}
For \(x\le1/2\), use \(h(x)\le\ln2\) and \(E(x)\le E(1/2)\) to get
\begin{equation}\label{eq:P0}
\varphi(x)\le P_0.
\end{equation}

For \(1/2<x<1\), away from the breakpoint \(z\),
\[
\frac{d}{dx}\ln\frac{\varphi(x)}x
=\frac{D'(x)}{2D(x)}+\frac{h'(x)}{2h(x)}+E'(x)-\frac1x.
\]
The second term is nonpositive, \(E'(x)\le E'(1/2)\), and
\(-1/x\le-1\).
On \([1/2,z]\), the first term is zero.
On \([z,1/\sqrt2]\), it equals
\[
x\frac{h'(x^2)}{h(x^2)}
\le\frac{h'(z^2)}{h(z^2)}.
\]
For this last bound, \(x\le1\), \(h'(x^2)\ge0\), and
\((h'/h)'=(h''h-(h')^2)/h^2<0\).
Above \(1/\sqrt2\), the first term is nonpositive.
Equation \eqref{eq:strip-checks} therefore proves that
\(\varphi(x)/x\) decreases on \([1/2,1)\). Continuity at \(z\)
and at one gives
\begin{equation}\label{eq:linear-profile}
\varphi(y)\le2P_0y\qquad(1/2\le y\le1).
\end{equation}

\subsection{The lower strip}
By symmetry, suppose \(x\le L\) and \(x\le y\).
If \(y\le1/2\), both coefficients in \eqref{eq:a} are one.
Thus \(r(x,y)=x\), and
\[
\varphi(x)\varphi(y)\le C_bP_0h(x)\le h(x)=K(x,y).
\]
If \(y\ge1/2\), then \(a(x)=1\), and
\[
r(x,y)=xk,\qquad k=y+a(y)(1-y)\in[y,1].
\]
Concavity \eqref{eq:concavity-scaling}, together with
\eqref{eq:lower-profile} and \eqref{eq:linear-profile}, gives
\[
K(x,y)=h(xk)\ge k h(x)\ge yh(x)
\ge\varphi(x)\varphi(y),
\]
where the last step uses \(2C_bP_0<1\).
This handles the whole lower strip, including \(x=0\).

\subsection{The upper strip}
For \(x\ge U\), both \(h(x)\) and \(h(x^2)\) are decreasing, since
\(U>1/\sqrt2\). Also \(1/x\le1/U\) and \(E(x)\le E(1)\).
Substituting these four bounds into \eqref{eq:phi} yields
\begin{equation}\label{eq:upper-linear}
\varphi(x)\le C_t x\qquad(x\ge U).
\end{equation}
For all \(0<y<1\) we claim
\[
\frac{D(y)}{h(y)}\le\frac{2\ln2}{h(z^2)}.
\]
For \(y\le1/2\), the left side is one.
For \(1/2\le y\le z\), use \(h(y)\ge h(z)\) and
\(h(z^2)\le2h(z)\), the latter being \eqref{eq:and}.
For \(y\ge z\), apply \eqref{eq:and} directly in \eqref{eq:D}.
It follows from \(E(y)\le E(1)\) that
\[
\varphi(y)\le C_g h(y).
\]
At the endpoints this holds by continuity. Since \(a(x)=0\) for
\(x\ge U\), we obtain
\[
\varphi(x)\varphi(y)\le C_tC_g\,xh(y)
\le xh(y)\le h(xy)=K(x,y).
\]
The last step is \eqref{eq:concavity-scaling}; the preceding step uses
\(C_tC_g<1\).
This handles the whole upper strip, even when \(y\) is in the
lower strip.

\subsection{The compact square}
To complete \eqref{eq:original-K}, it remains to check \([L,U]^2\).
The derivative of \(a(x)\) is singular at \(1/\sqrt2\), so we do
not differentiate it in the \(x\) coordinate. Instead, parameterize
the middle interval by \(t=a(x)\):
\begin{equation}\label{eq:middle-param}
x(t)=\frac1{t^2+\sqrt{(t^2-1)^2+1}},\qquad0\le t\le1.
\end{equation}
The defining equation for \(a\) rearranges to
\[
x^2+t^2x(1-x)=\frac12.
\]
Solving the quadratic and rationalizing gives
\eqref{eq:middle-param}, which remains valid at \(t=1\).
It moves continuously from \(1/\sqrt2\) to \(1/2\), with bounded
derivative. The lower and upper coordinate maps are
\[
x_{\mathrm{low}}(t)=L+(1/2-L)t,\qquad
x_{\mathrm{high}}(t)=1/\sqrt2+(U-1/\sqrt2)t.
\]
Their coefficients \(a\) are respectively one and zero.
The six unordered pairs of these three maps cover the square up
to symmetry. The interval verification proves that the gap is positive on
each of these six parameter squares.
Section~\ref{sec:certificates} describes the method, and
Appendix~\ref{app:derivatives} gives the derivative formulas.
Together with the two strips, this proves \eqref{eq:original-K}.

\subsection{The bridge}\label{sec:bridge}
For \(0\le y\le L\), equations \eqref{eq:concavity-scaling} and
\eqref{eq:lower-profile} imply
\[
h(Ay)-s(A)\varphi(y)\ge\bigl[A-s(A)C_b\bigr]h(y).
\]
The scalar checks give \(A-s(A)C_b>0\).
For \(L\le y\le A\), the one-variable interval verification proves
that \(h(Ay)-s(A)\varphi(y)>0\). This establishes
\eqref{eq:bridge} on the whole interval.
Section~\ref{sec:tail} now proves the stronger kernel
inequality \eqref{eq:K} everywhere.

\section{The mixture inequality}\label{sec:mixture}

It remains to prove \eqref{eq:M}.
This is where the chosen value \(m=0.61711475\) and the weights enter.
The profile \(\psi\) has already been shown to be admissible for
the sampling kernel.

\subsection{A change of variables}
Both sides of \eqref{eq:M} are symmetric in \(x,y\).
When \(0<x\le y\le1\) and \(xy<1\), write
\[
x=e^{-u},\quad y=e^{-v},\quad
u=t(1+\xi),\quad v=t(1-\xi),\quad t>0,\quad0\le\xi\le1.
\]
Explicitly, \(t=-\tfrac12\ln(xy)\) and
\(\xi=\ln(y/x)/[-\ln(xy)]\). Thus this change of variables covers
the entire relevant half-square.
Define
\begin{equation}\label{eq:GP-def}
G(u)=\frac{h(e^{-u})}{e^{-u}},\qquad
P(u)=\frac{\psi(e^{-u})}{e^{-u}},\qquad G(0)=P(0)=0.
\end{equation}
Here \(P(u)\) is a function, distinct from the probability
distributions \(P_i\) used in Section~\ref{sec:reduction}.
Dividing \eqref{eq:M} by \(2xy>0\) shows it is equivalent to
\begin{equation}\label{eq:F}
F(t,\xi)=mw_0G(2t)+mw_4P(u)P(v)
-\frac{G(u)+G(v)}2\ge0.
\end{equation}
It will be useful that
\begin{equation}\label{eq:k0}
k_0=mw_0=0.500000008685975>\frac12.
\end{equation}
The small positive excess over \(1/2\) controls the edge where one
of the original probabilities is one.

\subsection{Elementary estimates for \texorpdfstring{\(G\)}{G}}
Direct substitution gives
\[
G(u)=u-(e^u-1)\ln(1-e^{-u}).
\]
Using \(-\ln(1-r)=\sum_{k\ge1}r^k/k\) for \(0<r<1\),
and subtracting the two shifted series, yields
\begin{equation}\label{eq:Gseries}
G(u)=u+1-\sum_{j\ge1}\frac{e^{-ju}}{j(j+1)}.
\end{equation}
The series and its derivatives converge uniformly on compact
subintervals of \(u>0\). Therefore
\begin{equation}\label{eq:Gproperties}
\begin{gathered}
G'(u)=1+\sum_{j\ge1}\frac{e^{-ju}}{j+1}>0,\qquad
G''(u)=-\sum_{j\ge1}\frac{j\,e^{-ju}}{j+1}<0,\\
G(u)\ge u,\qquad
-G''(u)\le\sum_{j\ge1}e^{-ju}
=\frac1{e^u-1}\le\frac1u.
\end{gathered}
\end{equation}
For \(G(u)\ge u\), use
\(\sum_{j\ge1}1/[j(j+1)]=1\).

Set \(J(u)=G(u)+u\ln u\), with \(J(0)=0\).
Equation \eqref{eq:Gproperties} gives \(J''(u)\ge0\).
Also \(1-e^{-u}=u+O(u^2)\) and \(e^u-1=u+O(u^2)\), so
\[
G(u)=u-u\ln u+O(u^2|\ln u|),\qquad
\lim_{u\downarrow0}\frac{J(u)}u=1.
\]
Convexity implies that \(J(u)/u\) is nondecreasing, and hence
\begin{equation}\label{eq:Gsmall}
G(u)\ge u(1-\ln u).
\end{equation}
Convexity with \(J(0)=0\) also gives superadditivity:
if \(u,v\ge0\), then
\[
J(u)\le\frac{u}{u+v}J(u+v),\qquad
J(v)\le\frac{v}{u+v}J(u+v).
\]
Adding and rearranging proves
\begin{equation}\label{eq:Gdeficit}
G(u)+G(v)-G(u+v)
\le(u+v)h\left(\frac{u}{u+v}\right)\le2\sqrt{uv}.
\end{equation}
At \(u=v=0\), the statement follows by continuity.
For completeness, the last bound follows from concavity of \(\ln\):
\[
\frac{h(p)}2
=p\ln\frac1{\sqrt p}+(1-p)\ln\frac1{\sqrt{1-p}}
\le\ln\bigl(\sqrt p+\sqrt{1-p}\bigr).
\]
Multiplying by two and using \(\ln(1+r)\le r\) gives
\[
h(p)\le\ln\bigl(1+2\sqrt{p(1-p)}\bigr)
\le2\sqrt{p(1-p)}.
\]
Endpoint cases again follow by continuity.

\subsection{Small \texorpdfstring{\(t\)}{t}: the stronger tail supplies the gain}
Put \(T=0.001\). When \(t\le T\), both \(u,v\le2T\), and
\(e^{-u},e^{-v}\ge e^{-2T}>B\); the last inequality is included
in the scalar interval checks. On this range \(\psi=s\), so
\begin{equation}\label{eq:Ptail}
P(u)=\frac{\sqrt{h(e^{-2u})}}{e^{-u}}
=\sqrt{G(2u)}.
\end{equation}
Since \(u+v=2t\), equation \eqref{eq:Gdeficit} gives
\[
F(t,\xi)\ge(k_0-\tfrac12)G(2t)
+mw_4\sqrt{G(2u)G(2v)}-\sqrt{uv}.
\]
Using \eqref{eq:Gsmall} and \(2u,2v\le4T\), we have
\[
\sqrt{G(2u)G(2v)}
\ge2\sqrt{uv}\,[1-\ln(4T)].
\]
Thus
\begin{equation}\label{eq:smallt}
F(t,\xi)\ge(k_0-\tfrac12)G(2t)
+\bigl[2mw_4(1-\ln(4T))-1\bigr]\sqrt{uv}\ge0,
\end{equation}
because the scalar interval checks show that the bracket is positive. This proof includes \(u=0\) or \(v=0\) by limits.

\subsection{Large \texorpdfstring{\(t\)}{t}: a global lower bound for the profile}
Define
\begin{equation}\label{eq:Cmin}
C=\frac{\sigma e^{E(0)}}{\sqrt{h(x_*)}}.
\end{equation}
We claim \(D(x)\ge h(x)\). This is immediate for \(x\le z\).
For \(x\ge z\), the exact rational inequality \(z^2+z>1\)
implies \(1-x\le x^2\le x\). Because \(x\ge1/2\), symmetry and
concavity of \(h\) imply that its values throughout \([1-x,x]\)
are at least \(h(x)\). Consequently \(h(x^2)\ge h(x)\), and the
factor \(\ln2/h(z^2)\) in \(D\) is at least one.

Since \(E\) is increasing and \(\psi\ge\varphi\), we obtain
\[
\psi(x)\ge\varphi(x)\ge Ch(x),\qquad
P(u)\ge CG(u).
\]
For \(t\ge16\), one has \(u\ge t\), \(2t\ge u\), and \(G(u)\ge u\).
Using monotonicity of \(G\) in \eqref{eq:F},
\begin{align}
F(t,\xi)
&\ge (k_0-\tfrac12)G(u)
+\bigl[mw_4C^2G(u)-\tfrac12\bigr]G(v)\notag\\
&\ge0. \label{eq:larget}
\end{align}
The final step uses the scalar check
\[
16mw_4C^2-\frac12>0
\]
and \(G(u)\ge16\). This controls the entire unbounded tail of the
transformed domain.

\subsection{The compact rectangle and its difficult edge}
The remaining rectangle is
\[
[0.001,16]\times[0,1].
\]
The interval algorithm in Section~\ref{sec:certificates} proves
\(F>0\) on it. Two analytic estimates are used near \(\xi=1\),
where \(v=t(1-\xi)\) tends to zero and derivatives of \(G\) or
\(P\) are unbounded.

First, \(G(u)\le G(2t)\) and the product term is nonnegative, so
\begin{equation}\label{eq:edge1}
F(t,\xi)\ge(k_0-\tfrac12)G(2t)-\frac12G(v).
\end{equation}
Second, suppose
\begin{equation}\label{eq:edge-condition}
v\le-\ln B,\qquad G(v)\le[2mw_4CG(t)]^2.
\end{equation}
Then \eqref{eq:Ptail} gives
\(P(v)=\sqrt{G(2v)}\ge\sqrt{G(v)}\), while
\(P(u)\ge CG(u)\ge CG(t)\). Thus
\[
mw_4P(u)P(v)\ge mw_4CG(t)\sqrt{G(v)}
\ge \frac12G(v).
\]
This cancels the last term in \eqref{eq:edge1} and leaves
\begin{equation}\label{eq:edge2}
F(t,\xi)\ge(k_0-\tfrac12)G(2t)>0.
\end{equation}
At \(v=0\), the cancellation is the equality \(0=0\), so the
argument remains valid. These estimates avoid differentiating at
the singular edge.

The interval verification proves positivity throughout the remaining
compact rectangle, using these analytic estimates on boxes next to
the singular edge.
Together with \eqref{eq:smallt} and \eqref{eq:larget}, this proves
\(F\ge0\) for all \(t>0\), \(0\le\xi\le1\).
If \(xy=0\), both sides of \eqref{eq:M} vanish; if \(x=y=1\),
they also vanish. Symmetry handles \(x>y\).
We have therefore proved \eqref{eq:M} on the full unit square.
Proposition~\ref{prop:reduction}, with \(1-m=1531541/4000000\), now gives
Theorem~\ref{thm:main}.

\section{Interval verification}\label{sec:certificates}

A numerical grid checks only its sampled points. The calculations
here instead cover whole intervals or rectangles by bounds that
apply to every point in them. We describe the verification method here;
the derivative formulas appear in Appendix~\ref{app:derivatives}, and
the complete program appears in Appendix~\ref{app:code}.
The computation establishes the scalar signs used above, the compact
kernel inequality, the bridge inequality, and the compact mixture
inequality. These are inequalities over entire domains, rather than
numerical values of a root. Newton's method evaluates
\eqref{eq:stationary}, but does not establish these domain-wide
inequalities. The computational dependence of Theorem~\ref{thm:main}
is precisely this remaining verification. All four parts were completed with the exact parameters
in \eqref{eq:parameters}.

\subsection{Outward-rounded intervals}
An interval \([\ell,r]\) represents all real numbers between its
endpoints. Addition, subtraction, multiplication, and division
are evaluated with 35 significant decimal digits, rounding every
lower bound down and every upper bound up. Division is permitted
only when the denominator interval excludes zero.
For multiplication, all four endpoint products are considered.
For a logarithm, exponential, or square root, the correctly
rounded Decimal result is enlarged to the neighboring representable
value on each side. This uses the documented Decimal
guarantees~\cite{decimal}.

Every constant is supplied as an integer or a decimal string.
The center \(x_*=0.690908\) is supplied exactly; no root
isolation is needed for this parameter.
For \(h\) on an interval, concavity shows that its minimum is at
an endpoint and its maximum is either at an endpoint or at
\(1/2\), if \(1/2\) lies in the interval.
For \(G\), monotonicity allows evaluation at the two endpoints;
for \(G'\), the order is reversed because \(G''<0\).
All these monotonicity statements have been proved above.

\subsection{The acceptance test for a box}
Let \(f\) be a function on a box
\(I\times J\), and choose a representable center \((c,d)\).
Let \(\rho_I=\max(c-\inf I,\sup I-c)\), with an analogous
definition for \(\rho_J\). Suppose interval evaluation bounds
\(|f_x|\le M_x\) and \(|f_y|\le M_y\) throughout the box.
Integrating along horizontal and vertical segments gives
\begin{equation}\label{eq:meanvalue}
f(x,y)\ge f(c,d)-M_x\rho_I-M_y\rho_J
\quad((x,y)\in I\times J).
\end{equation}
The right side is evaluated with downward rounding.
The one-dimensional bridge uses the same formula with one term.

For a piecewise differentiable function, the segment can be
split at its finitely many breakpoints and the same bound applies
on each piece. The code intersects each interval with all relevant
pieces and takes the interval hull of their values and derivatives.
It includes both derivative branches at a join. The parameterized
kernel has no square-root derivative singularity, and mixture
boxes touching \(v=0\) are handled by the edge estimates instead.

For each pending box, the algorithm proceeds as follows:
\begin{enumerate}
\item For the mixture, first try the two analytic edge bounds.
\item Otherwise, evaluate the gap directly by interval arithmetic.
Accept if its lower endpoint is positive.
\item If necessary, evaluate the gap at the center and apply
\eqref{eq:meanvalue}. Accept if the resulting lower bound is positive.
\item If the center has a strictly negative upper bound, stop and
report a counterexample. If positivity is still unresolved, bisect
the box and retain \emph{both} children.
\end{enumerate}
The split uses the larger estimated derivative error when available,
or the larger coordinate width otherwise. This is an efficiency
choice; retaining both children is what preserves coverage.
Resource limits cause failure, not acceptance.
For mixture boxes with \(v=0\) not yet handled by an edge bound,
the code subdivides the \(\xi\) interval without differentiating.

The initial bridge interval is \([L,A]\).
The kernel starts with six unit parameter squares, one for each
unordered pair of the maps in Section~\ref{sec:kernel}.
The mixture starts with nine rectangles formed by the \(t\)-knots
\[
0.001,\ 0.01,\ 0.1,\ 0.3,\ 0.5,\ 1,\ 2,\ 4,\ 8,\ 16
\]
and the full \(\xi\)-interval \([0,1]\).
The stack is exhausted only after every descendant box has been
accepted. Hence a completed run proves full coverage.

The implementation is reproduced in full in Appendix~\ref{app:code}.
It uses only the Python standard library and is run without
optimization flags so that its assertions remain active. The analytic estimates outside the compact domains
are essential: the finite computation alone would not cover the
whole unit square. The proof of Theorem~\ref{thm:main} therefore consists
of the entropy reduction, those estimates, and this interval
verification together.

\section{A global coupling with a finite entropy gain}\label{sec:global}

We now couple entire sets, imposing only that both input marginals
are uniform on \(\F\). The construction uses relative-entropy
minimization and proves a quantitative gain over independent sampling.
Its proof does not use the interval computations of the preceding
sections.

The idea is to reduce the imbalance in how often different unions
occur. Under independent sampling, a set with many representations
as a union is more likely. Weighting each representation inversely
by that multiplicity makes the union uniform, but changes the input
marginals. We restore the uniform input marginals by a relative-entropy
projection. The proof compares this projection with an explicit small
perturbation of independent sampling that already preserves both
marginals. Removing row and column averages identifies an admissible
perturbation, and a four-entry calculation proves that its entropy gain
is strictly positive.

\subsection{Marginal constraints}
In the coordinatewise construction of Section~\ref{sec:reduction}, each next bit keeps its
native conditional law even after both input prefixes are revealed:
\[
\Pr(X_i=1\mid X_{<i},Y_{<i})=\Pr(X_i=1\mid X_{<i}).
\]
Here the native law already depends on the entire prefix.
For a global coupling, we impose only the full marginal conditions
\[
\Pr(X=A)=\Pr(Y=A)=\frac1N\qquad(A\in\F).
\]
The displayed conditional restriction is not imposed.
The joint law can use every element of both sets and the structure of
the entire family. General whole-vector couplings already occur in
Liu's formulation~\cite{liu}.

For the rest of the section assume \(N\ge2\), let \(T=\bigcup\F\in\F\),
and let \(P_0(A,B)=1/N^2\) be the independent coupling.
For \(S\in\F\), define
\[
r(S)=|\{(A,B)\in\F^2:A\cup B=S\}|,\qquad \nu_0(S)=\frac{r(S)}{N^2}.
\]
Every \(r(S)\) is positive because \((S,S)\) is a representation.
The pair counts sum to \(N^2\).
We count ordered pairs throughout this section.

\subsection{An auxiliary distribution and a balanced coupling}
Consider the strictly positive probability distribution
\[
Q(A,B)=\frac{1}{N\,r(A\cup B)}.
\]
One way to sample it is to select \(S\in\F\) uniformly and then select
one of its \(r(S)\) ordered representations uniformly.
Its union is uniform. Its two input marginals need not be uniform.

For probability distributions \(P,Q\) on a finite set, with \(Q>0\),
write
\[
D(P\Vert Q)=\sum_\omega P(\omega)\ln\frac{P(\omega)}{Q(\omega)},
\]
with a zero summand when \(P(\omega)=0\).
The inequality \(\ln t\le t-1\) gives \(D(P\Vert Q)\ge0\).
For example, applying it to \(t=Q(\omega)/P(\omega)\) on the
positive support of \(P\) gives
\[
D(P\Vert Q)\ge\sum_{P(\omega)>0}(P(\omega)-Q(\omega))\ge0.
\]

Let \(\mathcal C\) be the set of joint laws with both input marginals
uniform on \(\F\). Define
\[
P_*=\mathop{\rm argmin}_{P\in\mathcal C}D(P\Vert Q).
\]
This definition is an ordinary finite optimization over an \(N\)-by-\(N\)
matrix. The theorem below does not require enumerating all families or
numerically solving this optimization.

\begin{lemma}[Existence and the scaling formula]\label{lem:global-scaling}
The minimizer exists, is unique and strictly positive, and has the form
\[
P_*(A,B)=\frac{a(A)a(B)}{r(A\cup B)}
\]
for positive factors \(a(A)\) chosen to make every row and column sum
to \(1/N\). Thus it is the inverse-multiplicity coupling with
exponent exactly one.
\end{lemma}
\begin{proof}
The feasible set is a nonempty compact polytope: it contains \(P_0\).
The objective is continuous, using \(0\ln0=0\), so it has a minimum.
Strict convexity of \(t\ln t\) gives uniqueness.

The minimizer cannot have a zero entry. If it did, replace it by
\((1-s)P_*+sP_0\), which preserves both marginals.
At every zero entry, the change in \(P\ln(P/Q)\) contains a positive
constant times \(s\ln s\); at initially positive entries it is \(O(s)\).
For sufficiently small positive \(s\), the negative \(s\ln s\) terms
dominate, contradicting minimality.

At an interior minimum, Lagrange multipliers for the row and column
constraints give
\[
\ln\frac{P_*(A,B)}{Q(A,B)}+1=\alpha_A+\beta_B.
\]
One redundant marginal constraint may be deleted before applying the
multiplier rule. Exponentiating and absorbing constants yields
\(P_*(A,B)=u_Av_B/r(A\cup B)\), with positive \(u_A,v_B\).
Both the objective and the constraints are invariant under transpose,
so uniqueness gives \(P_*(A,B)=P_*(B,A)\).
Since all denominators are positive, \(u_Av_B=u_Bv_A\); hence
\(u_A/v_A\) is constant. Taking \(a(A)=\sqrt{u_Av_A}\) proves the formula.
\end{proof}

\subsection{Relative-entropy identities}
If \(P\in\mathcal C\) has union law \(\nu_P\), grouping the terms by
\(S=A\cup B\) gives
\begin{equation}\label{eq:KL-fibers}
D(P\Vert Q)
=D(\nu_P\Vert\operatorname{Unif}(\F))
 + \sum_S \nu_P(S)
 D\bigl(P(\,\cdot\mid A\cup B=S)\Vert
             \operatorname{Unif}(\text{representations of }S)\bigr).
\end{equation}
Indeed, inside each nonzero union fiber,
\[
\frac{P(A,B)}{Q(A,B)}
=\frac{\nu_P(S)}{1/N}\,
  \frac{P(A,B\mid S)}{1/r(S)}.
\]
Taking logarithms, multiplying by \(P(A,B)\), and summing proves
\eqref{eq:KL-fibers}. Fibers of zero \(\nu_P\)-mass contribute zero.
Consequently
\begin{equation}\label{eq:global-data-processing}
\ln N-H(\nu_P)\le D(P\Vert Q).
\end{equation}
For \(P=P_0\), the conditional law on each union fiber is uniform, so
\begin{equation}\label{eq:D0}
D(P_0\Vert Q)=\ln N-H(\nu_0).
\end{equation}

There is also an exact projection identity:
\begin{equation}\label{eq:global-pythagoras}
D(P_0\Vert Q)=D(P_0\Vert P_*)+D(P_*\Vert Q).
\end{equation}
To check it, subtract the right side from the left side. The result is
\[
\sum_{A,B}(P_0(A,B)-P_*(A,B))
             \ln\frac{P_*(A,B)}{Q(A,B)}.
\]
By the scaling formula, the logarithm is a sum of a function of \(A\),
a function of \(B\), and a constant. Equal row sums, equal column sums,
and equal total mass make the displayed expression zero.
In particular, \eqref{eq:global-data-processing}--\eqref{eq:global-pythagoras}
imply
\[
H(\nu_*)-H(\nu_0)\ge D(P_0\Vert P_*).
\]
The next argument gives a fully explicit lower bound.

\subsection{An explicit quantitative improvement}
Set
\[
c_{AB}=\ln r(A\cup B),\qquad
\bar c_A=\frac1N\sum_B c_{AB},\qquad
\bar c=\frac1N\sum_A\bar c_A,
\]
and define the matrix with its row and column means removed:
\[
R_{AB}=c_{AB}-\bar c_A-\bar c_B+\bar c,\qquad
V=\frac1{N^2}\sum_{A,B}R_{AB}^2.
\]
Symmetry of \(c\) explains why the row and column averages have the
same notation. Every row and column of \(R\) sums to zero.
With expectation taken under \(P_0\), these facts imply
\begin{equation}\label{eq:orthogonal}
\mathbb E_0R=0,\qquad \mathbb E_0(Rc)=\mathbb E_0(R^2)=V.
\end{equation}
For the second identity, substitute
\(c_{AB}=R_{AB}+\bar c_A+\bar c_B-\bar c\).
Each term involving a row mean, a column mean, or a constant vanishes.

\begin{theorem}[Finite gain for the global coupling]\label{thm:global-gain}
Let \(\nu_*\) be the union law of the coupling in
Lemma~\ref{lem:global-scaling}. Put \(\varepsilon=1/(8\ln N)\).
Then
\begin{equation}\label{eq:global-gain}
H(\nu_*)-H(\nu_0)
\ge \varepsilon(1-\varepsilon)V
\ge \frac{V}{16\ln N}
\ge \frac{\ln^2(2N-1)}{64N^2\ln N}>0.
\end{equation}
All logarithms and entropies in this formula use natural logarithms.
\end{theorem}
\begin{proof}
Since \(1\le r(S)\le N^2\), every \(c_{AB}\) lies between zero and
\(2\ln N\), as do its row and overall means. Therefore
\[
|R_{AB}|\le4\ln N.
\]
Consider the explicitly balanced comparison coupling
\[
P_\varepsilon(A,B)=\frac{1-\varepsilon R_{AB}}{N^2}.
\]
Its entries are at least \(1/(2N^2)\), and its row and column sums
are \(1/N\), by the zero row and column sums of \(R\).
Also \(0<\varepsilon<1/2\), since \(N\ge2\) and \(\ln2>1/2\).
The last elementary inequality follows by integrating \(1/t>1/2\)
over \(1<t<2\).

For any \(P\) and positive \(P_0\), the inequality \(\ln t\le t-1\)
also gives
\[
D(P\Vert P_0)
\le \sum_\omega\frac{(P(\omega)-P_0(\omega))^2}{P_0(\omega)}.
\]
To see the simplification, the immediate upper bound is
\(\sum_\omega P(\omega)(P(\omega)/P_0(\omega)-1)\);
expanding the square gives the same expression because both
distributions have total mass one.
For our comparison coupling this yields
\[
D(P_\varepsilon\Vert P_0)\le\varepsilon^2V.
\]
Moreover \(\ln(P_0(A,B)/Q(A,B))=c_{AB}-\ln N\). Consequently
\begin{align*}
D(P_\varepsilon\Vert Q)
&=D(P_\varepsilon\Vert P_0)
  +\mathbb E_\varepsilon(c-\ln N)\\
&\le \varepsilon^2V+\mathbb E_0(c-\ln N)
                -\varepsilon\mathbb E_0(Rc)\\
&=D(P_0\Vert Q)-\varepsilon(1-\varepsilon)V.
\end{align*}
Minimality gives \(D(P_*\Vert Q)\le D(P_\varepsilon\Vert Q)\).
Now use \eqref{eq:global-data-processing} and \eqref{eq:D0}.
This proves the first inequality in \eqref{eq:global-gain}.
Since \(1-\varepsilon\ge1/2\), the second follows as well.

To prove the remaining bound on \(V\), choose a minimal member
\(M\in\F\) under inclusion. Its only union representation is \((M,M)\),
so \(r(M)=1\). The largest member \(T\) has at least \(2N-1\)
representations: all pairs \((T,A)\) and \((A,T)\), counting \((T,T)\)
only once. Since \(N>1\), \(M\ne T\).
The row and column averages cancel in the four-entry difference:
\[
R_{MM}+R_{TT}-R_{MT}-R_{TM}
=c_{MM}+c_{TT}-c_{MT}-c_{TM}=-\ln r(T).
\]
The Cauchy--Schwarz inequality therefore gives
\[
\ln^2 r(T)
\le4\bigl(R_{MM}^2+R_{TT}^2+R_{MT}^2+R_{TM}^2\bigr)
\le4N^2V.
\]
Substitute \(r(T)\ge2N-1\) and combine the bounds.
\end{proof}

The comparison coupling itself also has a provable gain. If its union
law is \(\nu_\varepsilon\), write
\[
g(S)=\mathbb E_0(R\mid A\cup B=S),\qquad
\nu_\varepsilon(S)=\nu_0(S)(1-\varepsilon g(S)).
\]
Expanding relative entropy gives
\[
H(\nu_\varepsilon)-H(\nu_0)
=\varepsilon V-D(\nu_\varepsilon\Vert \nu_0).
\]
The same square bound, followed by
\((\mathbb E_0(R\mid S))^2\le\mathbb E_0(R^2\mid S)\), shows
\[
D(\nu_\varepsilon\Vert \nu_0)\le\varepsilon^2V.
\]
Thus its entropy increase is also at least
\(\varepsilon(1-\varepsilon)V\).
The affine formula gives an explicit alternative to solving the
matrix-scaling optimization.

\subsection{Infinitesimal entropy gain}
For \(\theta\) near zero, minimizing
\[
D(P\Vert P_0)+\theta\,\mathbb E_Pc
\qquad(P\in\mathcal C)
\]
produces the one-parameter family
\(P_\theta(A,B)=a_\theta(A)a_\theta(B)e^{-\theta c_{AB}}\).
Write \(\nu_\theta\) for its union law. At zero, \(a_0(A)=1/N\).
The log-scaling equations are differentiable there: their Jacobian
with respect to the log factors is
\(N^{-1}I+N^{-2}\mathbf1\mathbf1^{\mathsf T}\), which is positive
definite. The implicit function theorem thus justifies differentiation.
Writing \(b_A=(d/d\theta)\ln a_\theta(A)|_0\), differentiating the
row constraints gives
\[
b_A+\bar b=\bar c_A,\qquad
\bar b=\bar c/2,\qquad
\left.\frac{d}{d\theta}P_\theta(A,B)\right|_0=-\frac{R_{AB}}{N^2}.
\]
Differentiating union entropy, and using
\(\ln \nu_0(S)=\ln r(S)-2\ln N\), now gives
\[
\left.\frac{d}{d\theta}H(\nu_\theta)\right|_0
=\mathbb E_0(Rc)=V.
\]
Theorem~\ref{thm:global-gain} strengthens this infinitesimal statement:
the fixed exponent \(\theta=1\) has a finite, explicitly bounded gain.

\subsection{A consequence for element frequencies}
The global coupling can replace the independent component of the
entropy mixture from Section~\ref{sec:reduction}. This gives an additional correction depending
on the family size.

\begin{proposition}[A positive correction for each finite size]
\label{prop:global-frequency}
With the parameters of Theorem~\ref{thm:main}, put
\(c_0=0.38288525\) and \(k_0=mw_0=0.500000008685975\).
For a union-closed family of size \(N\ge2\), its maximum frequency
satisfies
\[
\boxed{\displaystyle
p\ \ge\ c_0+\frac{k_0V}{16(\ln N)^2}
\ \ge\ c_0+
\frac{k_0\ln^2(2N-1)}{64N^2(\ln N)^2}
\ >\ c_0.}
\]
Here \(V\) is the double-centered log-multiplicity variance defined
above.
\end{proposition}
\begin{proof}
Let \(Z_0\) and \(Z_4\) be the independent and shared-randomness unions
from Section~\ref{sec:reduction}; thus \(Z_0\) has law \(\nu_0\).
Let \(Z_*\) have the union law \(\nu_*\) of the global coupling.
All three take values in \(\F\).
The entropy reduction already proved
\[
w_0H(Z_0)+w_4H(Z_4)\ge \frac{1-p}{m}\ln N.
\]
By Theorem~\ref{thm:global-gain},
\(H(Z_*)\ge H(Z_0)+V/(16\ln N)\).
Weight this inequality by \(w_0\), then use the entropy upper bound
and \(w_0+w_4=1\):
\[
\ln N
\ge w_0H(Z_*)+w_4H(Z_4)
\ge \frac{1-p}{m}\ln N+\frac{w_0V}{16\ln N}.
\]
Multiply by \(m/\ln N>0\) and rearrange. This proves the first
inequality in the proposition. The estimate
\(V\ge\ln^2(2N-1)/(4N^2)\), proved above, supplies the second;
its rightmost correction is strictly positive.
\end{proof}

\subsection{The remaining quantitative question}
Proposition~\ref{prop:global-frequency} gives a positive correction for
each finite family size. Its size-only correction tends to zero as
\(N\) tends to infinity, so the estimate does not establish a larger
absolute constant than \(c_0\).
An extension would require a stronger relation between the entropy gain
and the element frequencies. In particular, the present estimates do
not determine whether the gain can remain comparable to \(\ln N\)
in families whose maximum frequency is small.

There is a useful exact check on this interpretation.
For any coupling with uniform input marginals, union closure gives
\(H(X\cup Y)\le\ln N\).
Equality forces the union to be uniform on \(\F\). Its expected
cardinality then equals that of \(X\).
Because \(X\subseteq X\cup Y\), the nonnegative difference of
cardinalities has expectation zero; hence \(X=X\cup Y\) almost surely.
Applying the same argument to \(Y\) gives \(X=Y\) almost surely.
Conversely that diagonal coupling achieves equality.
Since \(P_*\) is strictly positive and \(N>1\), it is not diagonal,
so its improved entropy still lies strictly below \(\ln N\).

\section*{Acknowledgments}
Part of this research was carried out at Axiom Math. The author
thanks Evan Chen, Jesse Thorner, Vasily Ilin, Simon Mahns, and Ken Ono
for fruitful discussions about AI for mathematics.
GPT-6 Astra served as an AI assistant in the research and preparation of
this manuscript.

\appendix
\section{Derivative formulas for the interval bounds}\label{app:derivatives}

This appendix makes the mean-value calculations independently
checkable. Each formula is an ordinary derivative on a smooth piece;
at joins, the code takes the hull of all applicable one-sided
derivative intervals.

\subsection{The profiles and the bridge}
The derivatives of \(E\) are
\[
E'(x)=-\frac{h'(x_*)}{2h(x_*)}-\frac75(x-x_*),
\qquad E''(x)=-\frac75.
\]
On the three pieces of \(D\), respectively,
\[
D'(x)=h'(x),\qquad D'(x)=0,\qquad
D'(x)=\frac{2x\ln2}{h(z^2)}h'(x^2).
\]
Logarithmic differentiation of \eqref{eq:phi} gives
\begin{equation}\label{eq:phiprime}
\varphi'(x)=\varphi(x)
\left[\frac{D'(x)}{2D(x)}+\frac{h'(x)}{2h(x)}+E'(x)\right].
\end{equation}
The compact checks avoid \(x=0,1\), where this quotient formula
is not used. Also
\[
s'(x)=\frac{x\,h'(x^2)}{s(x)}.
\]
On the transition interval,
\begin{align*}
\psi'(x)
&=(1-\theta(x))\varphi'(x)+\theta(x)s'(x)
+\frac{s(x)-\varphi(x)}{B-A}.
\end{align*}
Outside that interval, use \(\varphi'\) or \(s'\) as appropriate.
For the bridge gap \(b(y)=h(Ay)-s(A)\varphi(y)\),
\[
b'(y)=A h'(Ay)-s(A)\varphi'(y).
\]
These formulas are implemented by \texttt{uphi}, \texttt{sfun},
\texttt{psi}, and \texttt{bridge}. The name \texttt{uphi} denotes
the profile \(\varphi\) without a cutoff.

\subsection{The transformed mixture}
The derivative of \(G\) also has the convenient closed form
\[
G'(u)=-\frac{\ln(1-e^{-u})}{e^{-u}},
\]
obtained either by differentiation of \eqref{eq:GP-def} or by
summing the series for \(G'\).
For \(x=e^{-u}\), the chain rule gives
\[
P'(u)=\frac{\psi(x)}x-\psi'(x).
\]
In the code, the functions \texttt{GP} and \texttt{Psi} return
enclosures for \(G'\) and for \((P,P')\), respectively.

Put \(u=t(1+\xi)\), \(v=t(1-\xi)\).
Differentiating \eqref{eq:F} yields
\begin{align}
F_t={}&2k_0G'(2t)\notag\\
&+mw_4\bigl[(1+\xi)P'(u)P(v)
 +(1-\xi)P(u)P'(v)\bigr]\notag\\
&-\tfrac12\bigl[(1+\xi)G'(u)+(1-\xi)G'(v)\bigr],
\label{eq:Ft}\\
F_\xi={}&t\left\{
mw_4\bigl[P'(u)P(v)-P(u)P'(v)\bigr]
-\tfrac12\bigl[G'(u)-G'(v)\bigr]\right\}.
\label{eq:Fxi}
\end{align}
Every term on the right is evaluated over the entire box for the
mean-value bound. Breakpoints in \(\psi(e^{-u})\) are handled
internally by splitting the argument interval across the relevant
pieces of \(\psi\).

\subsection{The parameterized kernel}
For the middle coordinate map, let
\[
R_t=\sqrt{(t^2-1)^2+1},\qquad d_t=t^2+R_t.
\]
Then
\[
x(t)=d_t^{-1},\qquad a(x(t))=t,
\]
and
\begin{equation}\label{eq:coordinate-derivative}
x'(t)=-\frac{2t}{d_t^2}
\left(1+\frac{t^2-1}{R_t}\right),\qquad
\frac{d}{dt}a(x(t))=1.
\end{equation}
For the lower and upper maps, \(x'\) is simply their constant
slope and the derivative of \(a\) is zero.

For any pair of maps, use parameters \(\alpha,\beta\) and write
\[
x=x(\alpha),\quad a=a(x(\alpha)),\quad
y=y(\beta),\quad b=a(y(\beta)),\quad d=\min(x,y).
\]
The prime on \(a\) or \(b\) in the next formulas means differentiation
with respect to its parameter. On smooth pieces,
\begin{align*}
r_\alpha&=x'y+a'b(d-xy)+ab(d_\alpha-x'y),\\
r_\beta&=xy'+ab'(d-xy)+ab(d_\beta-xy').
\end{align*}
If \(x<y\), then \(d_\alpha=x'\) and \(d_\beta=0\).
If \(y<x\), then \(d_\alpha=0\) and \(d_\beta=y'\).
If the input intervals overlap, the code uses
\[
d_\alpha\in x'[0,1],\qquad d_\beta\in y'[0,1],
\]
which encloses both branches and any intervening one-sided choice.
For the gap
\(\Delta(\alpha,\beta)=h(r)-\varphi(x)\varphi(y)\), we obtain
\begin{align*}
\Delta_\alpha&=h'(r)r_\alpha-\varphi'(x)x'\varphi(y),\\
\Delta_\beta&=h'(r)r_\beta-\varphi(x)\varphi'(y)y'.
\end{align*}
Although natural interval evaluation may overestimate the range
of \(r\), its true range satisfies
\[
xy\le r\le\min(x,y).
\]
Indeed \(r\) is a convex combination of \(xy\) and \(\min(x,y)\)
with weight \(a(x)a(y)\in[0,1]\). On the compact kernel square
this implies
\[
\frac1{256}=L^2\le r\le U.
\]
The code intersects its interval for \(r\) with this proven range
before evaluating \(h\) and \(h'\). This narrowing preserves every
possible true value, while avoiding spurious domain violations.

\subsection{Bounds for an entire edge box}
For a mixture box \(t\in[t_-,t_+]\),
\(\xi\in[\xi_-,\xi_+]\), let
\[
v_{\max}=t_+(1-\xi_-).
\]
Since \(G\) is increasing, equation \eqref{eq:edge1} gives the
uniform lower bound
\[
(k_0-\tfrac12)G(2t_-)-\frac12G(v_{\max}).
\]
The stronger edge test verifies, with outward rounding,
\[
v_{\max}\le-\ln B,\qquad
G(v_{\max})\le[2mw_4CG(t_-)]^2.
\]
These inequalities imply \eqref{eq:edge-condition} at every
point in the box. Thus \eqref{eq:edge2} yields the uniform positive
lower bound \((k_0-\tfrac12)G(2t_-)\).
This explains the exact edge tests in the implementation.

\clearpage
\section{Complete interval verification}\label{app:code}

The following program supplies the computational part of
Theorem~\ref{thm:main}. Its four phases verify the scalar signs,
the compact kernel inequality, the bridge inequality, and the compact
mixture inequality, using the bounds in
Section~\ref{sec:certificates} and Appendix~\ref{app:derivatives}.
All four phases have completed successfully for the stated parameters.
No optimizer, data file, or external program is required.

To reproduce the computation, copy the listing into a file named
\texttt{verify.py} and run \texttt{python verify.py} with ordinary
Python~3. Do not use \texttt{-O}, and leave \texttt{PYTHONOPTIMIZE}
unset, since the assertions enforce domains and signs.
The program prints its results and also writes them as JSON files.
Those files are outputs of the computation, not inputs to the proof.
The completed computation used Python~3.12.13 with Decimal~1.70
and libmpdec~4.0.0.

The symbol \texttt{D} in the program is Python's Decimal class;
the mathematical function \(D(x)\) is evaluated inside
\texttt{uphi\_piece}. The function \texttt{uphi} implements
\(\varphi\), and \texttt{psi} implements \(\psi\).
The parameter \texttt{XCUT} is the exact breakpoint \(181/256\).

This computation proves inequalities over complete intervals and
rectangles. It is distinct from evaluating the decimal approximation
of the root in \eqref{eq:stationary}; the proof does not require a
certified numerical enclosure of that root.

\begin{lstlisting}
"""Interval verification for the union-closed frequency bound.

Run with Python 3, without -O and with PYTHONOPTIMIZE unset.
Only the Python standard library is required.
"""
from decimal import Decimal as D, Context, ROUND_FLOOR, ROUND_CEILING, getcontext
import json
import time
import hashlib

getcontext().prec = 60
PREC = 35
DOWN = Context(prec=PREC, rounding=ROUND_FLOOR)
UP = Context(prec=PREC, rounding=ROUND_CEILING)


class I:
    __slots__ = ('lo', 'hi')

    def __init__(self, lo, hi=None):
        if isinstance(lo, I):
            self.lo, self.hi = lo.lo, lo.hi
            return
        self.lo = lo if isinstance(lo, D) else D(str(lo))
        self.hi = self.lo if hi is None else (hi if isinstance(hi, D) else D(str(hi)))
        assert self.lo <= self.hi

    def __add__(self, other):
        b = I(other)
        return I(DOWN.add(self.lo,b.lo), UP.add(self.hi,b.hi))
    __radd__ = __add__

    def __neg__(self):
        return I(self.hi.copy_negate(), self.lo.copy_negate())

    def __sub__(self, other): return self + (-I(other))
    def __rsub__(self, other): return I(other) + (-self)

    def __mul__(self, other):
        b = I(other)
        pairs = [(a,c) for a in (self.lo,self.hi) for c in (b.lo,b.hi)]
        return I(min(DOWN.multiply(a,c) for a,c in pairs), max(UP.multiply(a,c) for a,c in pairs))
    __rmul__ = __mul__

    def __truediv__(self, other):
        b = I(other)
        assert b.lo > 0 or b.hi < 0, (b.lo,b.hi)
        return self * I(DOWN.divide(D(1),b.hi), UP.divide(D(1),b.lo))

    def __rtruediv__(self, other): return I(other)/self

    def exp(self):
        # Decimal exp/ln/sqrt are correctly rounded to nearest; enlarge one ulp.
        return I(DOWN.next_minus(DOWN.exp(self.lo)), UP.next_plus(UP.exp(self.hi)))

    def log(self):
        assert self.lo > 0
        return I(DOWN.next_minus(DOWN.ln(self.lo)), UP.next_plus(UP.ln(self.hi)))

    def sqrt(self):
        assert self.lo >= 0
        lo = D(0) if self.lo == 0 else DOWN.next_minus(DOWN.sqrt(self.lo))
        return I(lo, UP.next_plus(UP.sqrt(self.hi)))

    def absmax(self): return max(self.lo.copy_abs(),self.hi.copy_abs())

    def midpoint(self): return DOWN.divide(DOWN.add(self.lo,self.hi),D(2))

    def radius(self, center): return max(UP.subtract(center,self.lo),UP.subtract(self.hi,center))

    def width(self): return UP.subtract(self.hi,self.lo)

    def __repr__(self): return f'[{self.lo},{self.hi}]'


ZERO = I(0)
ONE = I(1)
HALF = I('.5')
LOG2 = I(2).log()


def hull(values): return I(min(x.lo for x in values),max(x.hi for x in values))


def hpoint(x):
    z = I(x)
    if z.lo == z.hi and z.lo in (D(0),D(1)): return ZERO
    assert z.lo > 0 and z.hi < 1, z
    return -(z*z.log()+(1-z)*(1-z).log())


def h(z):
    z=I(z)
    assert z.lo >= 0 and z.hi <= 1, z
    a,b=hpoint(z.lo),hpoint(z.hi)
    lower=min(a.lo,b.lo)
    upper=LOG2.hi if z.lo <= D('.5') <= z.hi else max(a.hi,b.hi)
    return I(lower,upper)


def hp(z): return (1-z).log()-z.log()


# Exact rational center; no algebraic-root approximation is used.
XSTAR = I('.690908')
HSTAR = h(XSTAR)
KSTAR = hp(XSTAR)/(2*HSTAR)
SIGMA = I('0.99999999')
XCUT = D(181)/D(256)
CUTH = h(I(XCUT)*I(XCUT))

def Gpoint(u):
    u=I(u)
    if u.lo == u.hi == 0: return ZERO
    x=(-u).exp()
    return h(x)/x


def G(u): return I(Gpoint(u.lo).lo,Gpoint(u.hi).hi)


def GPpoint(u):
    x=(-I(u)).exp()
    return -(1-x).log()/x


def GP(u):
    # G' decreases because G''<0, as follows from its convergent series.
    return I(GPpoint(u.hi).lo,GPpoint(u.lo).hi)



import sys

ZERO=I(0); ONE=I(1)
A=I('.716'); B=I('.721')
M=I('.61711475'); W0=I('.8102221'); W4=I('.1897779'); K0=M*W0
ROOT_HALF=1/I(2).sqrt()
LOW=D(1)/16; HIGH=D(15)/16
TMIN=D('.001'); TMAX=D(16)


def E(x):
    d=x-XSTAR
    return -KSTAR*d-I('.7')*d*d


def EP(x): return -KSTAR-I('1.4')*(x-XSTAR)


def uphi_piece(x,idx):
    hx=h(x)
    if idx==0: dx,dp=hx,hp(x)
    elif idx==1: dx,dp=LOG2,ZERO
    else: dx,dp=LOG2*h(x*x)/CUTH,LOG2*2*x*hp(x*x)/CUTH
    f=SIGMA*(dx*hx/HSTAR).sqrt()*E(x).exp()
    fp=f*(dp/(2*dx)+hp(x)/(2*hx)+EP(x))
    return f,fp


def uphi(x):
    # The input intervals used in compact certificates avoid 0 and 1.
    cuts=[D(0),D('.5'),XCUT,D(1)]
    vals=[];ders=[]
    for j in range(3):
        lo,hi=max(x.lo,cuts[j]),min(x.hi,cuts[j+1])
        if lo<=hi:
            f,fp=uphi_piece(I(lo,hi),j)
            vals.append(f);ders.append(fp)
    return hull(vals),hull(ders)


def sfun(x):
    s=h(x*x).sqrt()
    return s,x*hp(x*x)/s


def psi(x):
    vals=[];ders=[]
    for lo,hi,idx in [(max(x.lo,D(0)),min(x.hi,A.hi),0),
                       (max(x.lo,A.lo),min(x.hi,B.hi),1),
                       (max(x.lo,B.lo),min(x.hi,D(1)),2)]:
        if lo>hi: continue
        xx=I(lo,hi)
        if idx==0: f,fp=uphi(xx)
        elif idx==2: f,fp=sfun(xx)
        else:
            p,pp=uphi(xx);s,sp=sfun(xx)
            z=(xx-A)/(B-A)
            f=(1-z)*p+z*s
            fp=(1-z)*pp+z*sp+(s-p)/(B-A)
        vals.append(f);ders.append(fp)
    return hull(vals),hull(ders)


def Psi(u):
    x=(-u).exp();f,fp=psi(x)
    return f/x,f/x-fp


CMIN=SIGMA/HSTAR.sqrt()*E(I(0)).exp()
CBOTTOM=SIGMA/HSTAR.sqrt()*E(I(LOW)).exp()
LOWPHI=SIGMA*LOG2/HSTAR.sqrt()*E(I('.5')).exp()
CGLOBAL=SIGMA*(2*LOG2/(HSTAR*CUTH)).sqrt()*E(I(1)).exp()
CTOP=SIGMA/I(HIGH)*(LOG2/CUTH*h(I(HIGH)*I(HIGH))*h(I(HIGH))/HSTAR).sqrt()*E(I(1)).exp()


def scalar_checks():
    sa,_=sfun(A);pa,_=uphi(A)
    checks={
      'exponent_increasing':EP(I(1)),
      'top_extension_margin':1-CGLOBAL*CTOP,
      'bottom_low_margin':1-CBOTTOM*LOWPHI,
      'bottom_high_margin':1-2*CBOTTOM*LOWPHI,
      'phi_over_x_decreasing_margin':1-EP(I('.5'))-hp(I(XCUT)*I(XCUT))/CUTH,
      'bridge_small_y_margin':A-sa*CBOTTOM,
      'upper_tail_dominance_at_a':sa-pa,
      'upper_tail_ratio_decreasing_margin':-(hp(A)/(2*h(A))+EP(A)),
      'independent_edge_margin':K0-I('.5'),
      'small_t_margin':2*M*W4*(1-(4*I(TMIN)).log())-1,
      'small_t_uses_upper_tail':(-2*I(TMIN)).exp()-B,
      'large_t_margin':M*W4*CMIN*CMIN*I(TMAX)-I('.5')
    }
    for key,value in checks.items():
        assert value.lo>0,(key,value)
    return {k:{'lo':str(v.lo),'hi':str(v.hi)} for k,v in checks.items()}


def bridge(y,derivatives=False):
    p,pp=uphi(y);sa,_=sfun(A)
    f=h(A*y)-sa*p
    if derivatives:return f,A*hp(A*y)-sa*pp
    return f


def verify_bridge():
    stack=[I(LOW,A.hi)];n=0;minmargin=None;digest=hashlib.sha256()
    while stack:
        y=stack.pop();n+=1
        if n>100000: raise RuntimeError('Bridge resource limit')
        f,fp=bridge(y,True)
        mid=y.midpoint()
        if f.lo>0:bound=f.lo
        else:
            fc=bridge(I(mid))
            if fc.hi<0:raise ArithmeticError(('bridge negative witness',mid,fc))
            bound=DOWN.subtract(fc.lo,UP.multiply(fp.absmax(),y.radius(mid)))
        if bound>0:
            minmargin=bound if minmargin is None else min(minmargin,bound)
            digest.update(f'{y.lo},{y.hi},{bound}\n'.encode())
        else: stack.extend([I(y.lo,mid),I(mid,y.hi)])
    return {'status':'CERTIFIED','visited':n,'leaves':(n+1)//2,'smallest_lower':str(minmargin),'leaf_sha256':digest.hexdigest()}


def coord(t,branch):
    """Map [0,1] to low, middle, or high x ranges; return x,a,dx,da."""
    if branch==0:return I(LOW)+(I('.5')-I(LOW))*t,ONE,I('.5')-I(LOW),ZERO
    if branch==2:return ROOT_HALF+(I(HIGH)-ROOT_HALF)*t,ZERO,I(HIGH)-ROOT_HALF,ZERO
    z=t*t;r=((z-1)*(z-1)+1).sqrt();den=z+r
    x=1/den
    xp=-2*t*(1+(z-1)/r)/(den*den)
    return x,t,xp,ONE


def kernel_gap(t,z,bi,bj,derivatives=False):
    x,a,xp,ap=coord(t,bi);y,b,yp,bp=coord(z,bj)
    lower=I(min(x.lo,y.lo),min(x.hi,y.hi))
    if x.hi<=y.lo: lx,lz=xp,ZERO
    elif y.hi<=x.lo:lx,lz=ZERO,yp
    else: lx,lz=xp*I(0,1),yp*I(0,1)
    q=x*y+a*b*(lower-x*y)
    # For all real arguments: xy<=q<=min(x,y).
    q=I(max(q.lo,D(1)/256),min(q.hi,HIGH))
    px,pdx=uphi(x);py,pdy=uphi(y)
    val=h(q)-px*py
    if not derivatives:return val
    qt=xp*y+ap*b*(lower-x*y)+a*b*(lx-xp*y)
    qz=x*yp+a*bp*(lower-x*y)+a*b*(lz-x*yp)
    return val,hp(q)*qt-pdx*xp*py,hp(q)*qz-px*pdy*yp


def verify_kernel():
    start=time.monotonic();last=start
    stack=[(I(0,1),I(0,1),i,j,0) for i in range(3) for j in range(i,3)]
    n=leaves=0;minmargin=None;maxdepth=0;digest=hashlib.sha256()
    while stack:
        t,z,i,j,depth=stack.pop();n+=1;maxdepth=max(maxdepth,depth)
        if n>500000 or depth>65:raise RuntimeError(('kernel unresolved',n,depth,t,z,i,j))
        try:val,dt,dz=kernel_gap(t,z,i,j,True)
        except AssertionError:val=dt=dz=None
        tm,zm=t.midpoint(),z.midpoint()
        if val is not None and val.lo>0:bound=val.lo
        else:
            fc=kernel_gap(I(tm),I(zm),i,j)
            if fc.hi<0:raise ArithmeticError(('kernel negative witness',i,j,tm,zm,fc))
            et=UP.multiply(dt.absmax(),t.radius(tm)) if dt else None
            ez=UP.multiply(dz.absmax(),z.radius(zm)) if dz else None
            bound=DOWN.subtract(DOWN.subtract(fc.lo,et),ez) if dt else D('-Infinity')
        if bound>0:
            leaves+=1;minmargin=bound if minmargin is None else min(minmargin,bound)
            digest.update(f'{i},{j},{t.lo},{t.hi},{z.lo},{z.hi},{bound}\n'.encode())
        else:
            if dt and et>=ez or (not dt and t.width()>=z.width()):
                stack.extend([(I(t.lo,tm),z,i,j,depth+1),(I(tm,t.hi),z,i,j,depth+1)])
            else:stack.extend([(t,I(z.lo,zm),i,j,depth+1),(t,I(zm,z.hi),i,j,depth+1)])
        now=time.monotonic()
        if now-last>15:
            print(json.dumps({'phase':'kernel','visited':n,'leaves':leaves,'pending':len(stack),'seconds':round(now-start,1)}),flush=True);last=now
    return {'status':'CERTIFIED','visited':n,'leaves':leaves,'max_depth':maxdepth,'smallest_lower':str(minmargin),'leaf_sha256':digest.hexdigest(),'seconds':round(time.monotonic()-start,2)}


def mixture(t,z,derivatives=False):
    u,v=t*(1+z),t*(1-z)
    pu,pup=Psi(u);pv,pvp=Psi(v)
    val=M*(W0*G(2*t)+W4*pu*pv)-(G(u)+G(v))/2
    if not derivatives:return val
    gu,gv=GP(u),GP(v)
    dt=2*K0*GP(2*t)+M*W4*((1+z)*pup*pv+(1-z)*pu*pvp)-((1+z)*gu+(1-z)*gv)/2
    dz=t*(M*W4*(pup*pv-pu*pvp)-(gu-gv)/2)
    return val,dt,dz


def verify_mixture():
    start=time.monotonic();last=start
    knots=list(map(D,['.001','.01','.1','.3','.5','1','2','4','8','16']))
    stack=[(I(a,b),I(0,1),0) for a,b in zip(knots[:-1],knots[1:])]
    n=leaves=edges=0;minmargin=None;maxdepth=0;digest=hashlib.sha256()
    tail_limit=-B.log()
    while stack:
        t,z,depth=stack.pop();n+=1;maxdepth=max(maxdepth,depth)
        if n>500000 or depth>65:raise RuntimeError(('mixture unresolved',n,depth,t,z))
        vmax=UP.multiply(t.hi,UP.subtract(D(1),z.lo))
        vg=G(I(vmax))
        bound_edge=(K0-I('.5'))*G(2*I(t.lo))-vg/2
        edge2=(2*M*W4*CMIN*G(I(t.lo)))
        if bound_edge.lo>0:
            bound=bound_edge.lo;edges+=1
        elif vmax<=tail_limit.lo and vg.hi<=(edge2*edge2).lo:
            bound=((K0-I('.5'))*G(2*I(t.lo))).lo;edges+=1
        else:
            if (t*(1-z)).lo<=0:
                mid=z.midpoint();stack.extend([(t,I(z.lo,mid),depth+1),(t,I(mid,z.hi),depth+1)]);continue
            try:val,dt,dz=mixture(t,z,True)
            except AssertionError:val=dt=dz=None
            tm,zm=t.midpoint(),z.midpoint()
            if val is not None and val.lo>0:bound=val.lo
            else:
                fc=mixture(I(tm),I(zm))
                if fc.hi<0:raise ArithmeticError(('mixture negative witness',tm,zm,fc))
                et=UP.multiply(dt.absmax(),t.radius(tm)) if dt else None
                ez=UP.multiply(dz.absmax(),z.radius(zm)) if dz else None
                bound=DOWN.subtract(DOWN.subtract(fc.lo,et),ez) if dt else D('-Infinity')
                if bound<=0:
                    if dt and et>=ez or (not dt and t.width()>=z.width()):
                        stack.extend([(I(t.lo,tm),z,depth+1),(I(tm,t.hi),z,depth+1)])
                    else:stack.extend([(t,I(z.lo,zm),depth+1),(t,I(zm,z.hi),depth+1)])
                    continue
        leaves+=1;minmargin=bound if minmargin is None else min(minmargin,bound)
        digest.update(f'{t.lo},{t.hi},{z.lo},{z.hi},{bound}\n'.encode())
        now=time.monotonic()
        if now-last>15:
            print(json.dumps({'phase':'mixture','visited':n,'leaves':leaves,'pending':len(stack),'seconds':round(now-start,1)}),flush=True);last=now
    return {'status':'CERTIFIED','visited':n,'leaves':leaves,'edge_leaves':edges,'max_depth':maxdepth,'smallest_lower':str(minmargin),'leaf_sha256':digest.hexdigest(),'seconds':round(time.monotonic()-start,2)}


def main():
    selected=sys.argv[1:] or ['scalars','bridge','kernel','mixture']
    functions={'scalars':scalar_checks,'bridge':verify_bridge,'kernel':verify_kernel,'mixture':verify_mixture}
    for name in selected:
        result=functions[name]()
        with open(f'frankl_038288525_{name}.json','w') as f:json.dump(result,f,indent=2)
        print(json.dumps({'phase':name,'result':result}),flush=True)


if __name__=='__main__':main()
\end{lstlisting}

\end{document}